\documentclass[11pt]{amsart}

\usepackage{amssymb}
\usepackage{amsmath,amssymb,epsf,wrapfig,multicol,epic,eepic,trig,mathrsfs,dsfont,color,graphicx}
 \usepackage[abs]{overpic}
 \usepackage{here}

\theoremstyle{plain}
\newtheorem{thm}{Theorem}[section]
\newtheorem{cor}[thm]{Corollary}

\newtheorem{prop}[thm]{Proposition}

\newtheorem{definition}[thm]{Definition}
\newtheorem{rem}[thm]{Remark}
\newtheorem{exa}[thm]{Example}

\def\R{{\mathbb R}}

\def\rmoveiod#1#2{
\setlength{\unitlength}{#1}
\begin{picture}(50,30)
\put(5,0){\line(0,1){30}}

{\allinethickness{.8pt}
\put(10,15){\vector(1,0){15}}
\put(25,15){\vector(-1,0){15}}}

\qbezier(30,0)(30,20)(45,20)
\qbezier(45,20)(50,20)(50,15)
\qbezier(50,15)(50,10)(45,10)
\qbezier(45,10)(40,10)(36,14)
\qbezier(36,14)(30,24)(30,30)

\ifnum#2=2
\put(3,28){\path(0,0)(2,2)(4,0)}
\put(28,28){\path(0,0)(2,2)(4,0)}
\put(0,0){\makebox{${\Huge \overline{x}}$}}
\put(7,0){\makebox{${\Huge \underline{x}}$}}
\put(25,0){\makebox{${\Huge \overline{x}}$}}
\put(32,0){\makebox{$\Huge \underline{x}$}}
\put(45,6){\makebox{${ \overline{y}}$}}
\put(45,13){\makebox{${\Huge \underline{y}}$}}
\put(23,26){\makebox{$\Huge \underline{y}^{\overline{x}}$}}
\put(33,26){\makebox{$\Huge \overline{y}_{\underline{x}}$}}
\put(38,15){\makebox{${ \overline{x}_{\underline{y}}}$}}
\put(38,21){\makebox{${\Huge \underline{x}^{\overline{y}}}$}}
\fi

\end{picture}
}

\def\rmoveiiod#1#2{
\setlength{\unitlength}{#1}
\begin{picture}(80,40)
\put(5,0){\line(0,1){35}}
\put(15,0){\line(0,1){35}}

{\allinethickness{.8pt}
\put(20,17){\vector(1,0){20}}
\put(40,17){\vector(-1,0){20}}}

\qbezier(50,0)(90,17.5)(50,35)
\qbezier(70,0)(30,17.5)(70,35)

\ifnum#2=2
\put(2,32){\path(0,0)(3,3)(6,0)}
\put(12,2){\path(0,0)(3,-3)(6,0)}
\put(50,32){\path(0,0)(0,3)(3,3)}
\put(70,3){\path(0,0)(0,-3)(-3,-3)}
\put(0,0){\makebox{$\overline{x}$}}
\put(7,0){\makebox{$\underline{x}$}}
\put(8,35){\makebox{$\overline{y}_{\underline{x}}$}}
\put(17,35){\makebox{$\underline{y}^{\overline{x}}$}}

\put(46,-2){\makebox{$\overline{x}$}}
\put(55,-1){\makebox{$\underline{x}$}}
\put(61,12){\makebox{$\underline{x}^{\overline{z}}$}}
\put(71,12){\makebox{$\overline{x}_{\underline{z}}$}}
\put(40,38){\makebox{$\overline{x}_{\underline{z}}^{\phantom{y}\underline{y}^{\overline{x}}}$}}
\put(51,38){\makebox{$\underline{x}^{\overline{z}}_{\phantom{y}\overline{y}_{\underline{x}}}$}}

\put(64,38){\makebox{$\overline{y}_{\underline{x}}$}}
\put(73,38){\makebox{$\underline{y}^{\overline{x}}$}}

\put(38,18){\makebox{$\underline{y}^{\overline{x}}_{\phantom{x}\overline{x}_{\underline{z}}}$}}
\put(54,18){\makebox{$\overline{y}_{\underline{x}}^{\phantom{x}\underline{x}^{\overline{z}}}$}}
\put(56,7){\makebox{$\overline{z}$}}
\put(49,7){\makebox{$\underline{z}$}}
\put(61,-3){\makebox{$\overline{z}_{\underline{x}}$}}
\put(71,-3){\makebox{$\underline{z}^{\overline{x}}$}}

\fi

\ifnum#2=3
\put(2,3){\path(0,0)(3,-3)(6,0)}
\put(12,32){\path(0,0)(3,3)(6,0)}
\put(50,3){\path(0,0)(0,-3)(3,-3)}
\put(70,32){\path(0,0)(0,3)(-3,3)}

\put(0,35){\makebox{$\underline{x}$}}
\put(7,35){\makebox{$\overline{x}$}}
\put(8,0){\makebox{$\underline{y}^{\overline{x}}$}}
\put(17,0){\makebox{$\overline{y}_{\underline{x}}$}}

\put(46,35){\makebox{$\underline{x}$}}
\put(55,35){\makebox{$\overline{x}$}}
\put(62,18){\makebox{$\overline{x}_{\underline{z}}$}}
\put(72,18){\makebox{$\underline{x}^{\overline{z}}$}}
\put(38,-2){\makebox{$\underline{x}^{\overline{z}}_{\phantom{y}\overline{y}_{\underline{x}}}$}}
\put(51,-2){\makebox{$\overline{x}_{\underline{z}}^{\phantom{y}\underline{y}^{\overline{x}}}$}}
\put(63,36){\makebox{$\underline{z}^{\overline{x}}$}}
\put(73,36){\makebox{$\overline{z}_{\underline{x}}$}}

\put(50,25){\makebox{$\overline{z}$}}
\put(58,25){\makebox{$\underline{z}$}}
\put(54,12){\makebox{$\underline{y}^{\overline{x}}_{\phantom{x}\overline{x}_{\underline{z}}}$}}
\put(37,12){\makebox{$\overline{y}_{\underline{x}}^{\phantom{x}\underline{x}^{\overline{z}}}$}}
\put(62,-5){\makebox{$\underline{y}^{\overline{x}}$}}
\put(72,0){\makebox{$\overline{y}_{\underline{x}}$}}

\fi

\end{picture}
}
\def\rmoveiiiod#1#2{
\setlength{\unitlength}{#1}
\begin{picture}(75,35)
\put(5,0){\line(1,1){15}}
\qbezier(20,15)(25,20)(25,30)

\put(15,0){\line(-1,1){7}}
\qbezier(9,6)(0,15)(10,25)
\put(10,25){\line(1,1){5}}

\qbezier(25,0)(25,10)(21,14)
\put(21,14){\line(-1,1){16}}

{\allinethickness{.8pt}
\put(32,15){\vector(1,0){15}}
\put(47,15){\vector(-1,0){15}}}

\qbezier(55,0)(55,10)(60,15)
\put(60,15){\line(1,1){15}}

\put(65,0){\line(1,1){5}}
\qbezier(70,5)(80,15)(71,24)
\put(71,24){\line(-1,1){6}}

\put(75,0){\line(-1,1){16}}
\qbezier(59,16)(55,20)(55,30)

\ifnum#2=2
\put(5,27){\path(0,0)(0,3)(3,3)}
\put(12,30){\path(0,0)(3,0)(3,-3)}
\put(22,27){\path(0,0)(3,3)(6,0)}

\put(1,-2){\makebox{$\overline{x}$}}
\put(5,-2){\makebox{$\underline{x}$}}
\put(11,-2){\makebox{$\overline{y}$}}
\put(16,-2){\makebox{$\underline{y}$}}
\put(22,-2){\makebox{$\overline{z}$}}
\put(25,-2){\makebox{$\underline{z}$}}

\put(52,27){\path(0,0)(3,3)(6,0)}
\put(65,27){\path(0,0)(0,3)(3,3)}
\put(72,30){\path(0,0)(3,0)(3,-3)}

\fi

\end{picture}
}
\def\rmovevio#1#2{
\setlength{\unitlength}{#1}
\begin{picture}(50,35)
\put(5,5){\line(0,1){30}}

{\allinethickness{.8pt}
\put(10,20){\vector(1,0){15}}
\put(25,20){\vector(-1,0){15}}}

\qbezier(30,5)(30,25)(45,25)
\qbezier(45,25)(50,25)(50,20)
\qbezier(50,20)(50,15)(45,15)
\qbezier(45,15)(30,15)(30,35)

\put(34,20){\circle{5}}

\ifnum#2=2
\put(3,32){\path(0,0)(2,2)(4,0)}
\put(28,32){\path(0,0)(2,2)(4,0)}
\put(38,20){\makebox{${\Huge c_{1}}$}}
\fi

\end{picture}
}

\def\rmoveviio#1#2{
\setlength{\unitlength}{#1}
\begin{picture}(70,40)
\put(5,5){\line(0,1){35}}
\put(15,5){\line(0,1){35}}

{\allinethickness{.8pt}
\put(20,20){\vector(1,0){20}}
\put(40,20){\vector(-1,0){20}}}

\qbezier(47,5)(87,22.5)(47,40)
\qbezier(67,5)(27,22.5)(67,40)
\put(57,10){\circle{5}}
\put(57,35){\circle{5}}

\ifnum#2=2
\put(2,37){\path(0,0)(3,3)(6,0)}
\put(12,37){\path(0,0)(3,3)(6,0)}
\put(47,37){\path(0,0)(0,3)(3,3)}
\put(67,37){\path(0,0)(0,3)(-3,3)}
\put(52,25){\makebox{${\Huge c_{1}}$}}
\put(52,13){\makebox{${\Huge c_{2}}$}}
\fi

\ifnum#2=3
\put(2,7){\path(0,0)(3,-3)(6,0)}
\put(12,37){\path(0,0)(3,3)(6,0)}
\put(47,8){\path(0,0)(0,-3)(3,-3)}
\put(67,37){\path(0,0)(0,3)(-3,3)}
\put(54,25){\makebox{${\Huge c_{1}}$}}
\put(54,13){\makebox{${\Huge c_{2}}$}}
\fi

\end{picture}
}

\def\rmoveviiio#1#2{
\setlength{\unitlength}{#1}
\begin{picture}(70,35)
\put(0,5){\line(1,1){15}}
\qbezier(15,20)(20,25)(20,35)

\put(10,5){\line(-1,1){5}}
\qbezier(5,10)(-5,20)(5,30)
\put(5,30){\line(1,1){5}}

\qbezier(20,5)(20,15)(15,20)
\put(15,20){\line(-1,1){15}}

\put(5,10){\circle{5}}
\put(15,20){\circle{5}}
\put(5,30){\circle{5}}

{\allinethickness{.8pt}
\put(27,20){\vector(1,0){15}}
\put(42,20){\vector(-1,0){15}}}

\qbezier(50,5)(50,15)(55,20)
\put(55,20){\line(1,1){15}}

\put(60,5){\line(1,1){5}}
\qbezier(65,10)(75,20)(65,30)
\put(65,30){\line(-1,1){5}}

\put(70,5){\line(-1,1){15}}
\qbezier(55,20)(50,25)(50,35)

\put(65,10){\circle{5}}
\put(55,20){\circle{5}}
\put(65,30){\circle{5}}

\ifnum#2=2
\put(0,32){\path(0,0)(0,3)(3,3)}
\put(7,35){\path(0,0)(3,0)(3,-3)}
\put(17,32){\path(0,0)(3,3)(6,0)}

\put(20,18){\makebox{${\Huge c_{1}}$}}

\put(47,32){\path(0,0)(3,3)(6,0)}
\put(60,32){\path(0,0)(0,3)(3,3)}
\put(67,35){\path(0,0)(3,0)(3,-3)}

\put(60,18){\makebox{${\Huge c'_{1}}$}}
\fi

\end{picture}
}
\def\rmovevivod#1#2{
\setlength{\unitlength}{#1}
\begin{picture}(75,40)
\put(5,5){\line(1,1){15}}
\qbezier(20,20)(25,25)(25,35)

\put(15,5){\line(-1,1){7}}
\qbezier(9,11)(0,20)(10,30)
\put(10,30){\line(1,1){5}}

\qbezier(25,5)(25,15)(21,19)
\put(21,19){\line(-1,1){16}}

\put(10,10){\circle{5}}
\put(10,30){\circle{5}}

{\allinethickness{.8pt}
\put(33,20){\vector(1,0){15}}
\put(48,20){\vector(-1,0){15}}}

\qbezier(55,5)(55,15)(60,20)
\put(60,20){\line(1,1){15}}

\put(65,5){\line(1,1){5}}
\qbezier(70,10)(80,20)(71,29)
\put(71,29){\line(-1,1){6}}

\put(75,5){\line(-1,1){16}}
\qbezier(59,21)(55,25)(55,35)

\put(70,10){\circle{5}}
\put(70,30){\circle{5}}

\ifnum#2=2
\put(5,7){\path(0,0)(0,-2)(2,-2)}
\put(13,5){\path(0,0)(2,0)(2,2)}
\put(23,7){\path(0,0)(2,-2)(4,0)}

\put(0,3){\makebox{$\overline{z}_{\underline{x}}$}}
\put(7,3){\makebox{$\underline{z}^{\overline{x}}$}}
\put(19,3){\makebox{$\overline{x}_{\underline{z}}$}}
\put(27,3){\makebox{$\underline{x}^{\overline{z}}$}}

\put(21,36){\makebox{$\underline{z}$}}
\put(26,36){\makebox{$\overline{z}$}}
\put(11,36){\makebox{$\underline{y}$}}
\put(16,36){\makebox{$\overline{y}$}}
\put(1,36){\makebox{$\underline{x}$}}
\put(6,36){\makebox{$\overline{x}$}}

\put(53,7){\path(0,0)(2,-2)(4,0)}
\put(65,7){\path(0,0)(0,-2)(2,-2)}
\put(73,5){\path(0,0)(2,0)(2,2)}

\put(50,3){\makebox{$\overline{z}_{\underline{x}}$}}
\put(56,3){\makebox{$\underline{z}^{\overline{x}}$}}
\put(70,3){\makebox{$\overline{x}_{\underline{z}}$}}
\put(76,3){\makebox{$\underline{x}^{\overline{z}}$}}

\put(51,36){\makebox{$\underline{x}$}}
\put(56,36){\makebox{$\overline{x}$}}
\put(62,36){\makebox{$\underline{y}$}}
\put(67,36){\makebox{$\overline{y}$}}
\put(72,36){\makebox{$\underline{z}$}}
\put(76,36){\makebox{$\overline{z}$}}

\fi
\end{picture}
}

\def\rmovevivodold#1#2{
\setlength{\unitlength}{#1}
\begin{picture}(70,30)
\put(0,0){\line(1,1){15}}
\qbezier(15,15)(20,20)(20,30)

\put(10,0){\line(-1,1){5}}
\qbezier(5,5)(-5,15)(5,25)
\put(5,25){\line(1,1){5}}

\qbezier(20,0)(20,10)(16,14)
\put(16,14){\line(-1,1){16}}

\put(5,5){\circle{5}}
\put(5,25){\circle{5}}

{\allinethickness{.8pt}
\put(30,15){\vector(1,0){10}}
\put(40,15){\vector(-1,0){10}}}

\qbezier(50,0)(50,10)(55,15)
\put(55,15){\line(1,1){15}}

\put(60,0){\line(1,1){5}}
\qbezier(65,5)(75,15)(65,25)
\put(65,25){\line(-1,1){5}}

\put(70,0){\line(-1,1){16}}
\qbezier(54,16)(50,20)(50,30)

\put(65,5){\circle{5}}
\put(65,25){\circle{5}}

\ifnum#2=2
\put(0,27){\path(0,0)(0,3)(3,3)}
\put(7,30){\path(0,0)(3,0)(3,-3)}
\put(17,27){\path(0,0)(3,3)(6,0)}

\put(20,13){\makebox{${\Huge c_{1}}$}}

\put(47,27){\path(0,0)(3,3)(6,0)}
\put(60,27){\path(0,0)(0,3)(3,3)}
\put(67,30){\path(0,0)(3,0)(3,-3)}

\put(60,13){\makebox{${\Huge c'_{1}}$}}
\fi

\end{picture}
}

\def\rmovetiiio#1#2{
\setlength{\unitlength}{#1}
\begin{picture}(70,20)

\put(0,0){\line(1,1){20}}
\put(20,0){\line(-1,1){9}}
\put(9,11){\line(-1,1){9}}

\put(7,3){\line(-1,1){4}}
\put(13,3){\line(1,1){4}}
\put(3,13){\line(1,1){4}}
\put(17,13){\line(-1,1){4}}

{\allinethickness{.8pt}
\put(25,10){\vector(1,0){10}}
\put(35,10){\vector(-1,0){10}}}

\qbezier(40,5)(48,20)(54,11)
\qbezier(56, 9)(62,0)(70,15)

\qbezier(40,15)(48,0)(55,10)
\qbezier(55,10)(62,20)(70,5)

\put(43,10){\circle{5}}
\put(67,10){\circle{5}}

\ifnum#2=2
\put(0,17){\path(0,0)(0,3)(3,3)}
\put(17,20){\path(0,0)(3,0)(3,-3)}

\put(15,8){\makebox{${\Huge c_{1}}$}}

\put(40,12){\path(0,0)(0,3)(3,3)}
\put(67,15){\path(0,0)(3,0)(3,-3)}

\put(53,3){\makebox{${\Huge c'_{1}}$}}
\fi

\end{picture}
}

\begin{document}

\title[Colorings and doubled colorings of virtual doodles]
{Colorings and doubled colorings of virtual doodles}

\author[A.~Bartholomew]{Andrew Bartholomew}
\author[R.~Fenn]{Roger Fenn}
\author[N.~Kamada]{Naoko Kamada}
\author[S.~Kamada]{Seiichi Kamada} 

\thanks{This work was supported by JSPS KAKENHI Grant Numbers JP15K04879 and JP26287013.}
\address[A.~Bartholomew]{ 
School of Mathematical Sciences, University of Sussex\\
Falmer, Brighton, BN1 9RH, England}
\email{andrewb@layer8.co.uk}
\address[R.~Fenn]{ School of Mathematical Sciences, University of Sussex\\
Falmer, Brighton, BN1 9RH, England}
\email{rogerf@sussex.ac.uk}
\address[N.~Kamada]{Graduate School of Natural Sciences,  Nagoya City University \\
1 Yamanohata, Mizuho-cho, Mizuho-ku, Nagoya, Aichi 467-8501 Japan}
\email{kamada@nsc.nagoya-cu.ac.jp}
\address[S.~Kamada]{Department of Mathematics, Osaka City University\\
Suimiyoshi,  Osaka 558-8585, Japan}
\email{skamada@sci.osaka-cu.ac.jp}

\date{\today}

\begin{abstract} 
A virtual doodle is an equivalence class of virtual diagrams under an equivalence relation generated by flat version of classical Reidemesiter moves and virtual Reidemsiter  moves such that Reidemeister moves of type 3 are forbidden.  
In this paper we discuss colorings of virtual diagrams using an algebra, called a doodle switch, and define an invariant of virtual doodles. Besides usual colorings of diagrams, we also introduce doubled colorings. 
\end{abstract}

\maketitle


\section{Introduction}

A {\it virtual diagram} is a generically immersed closed $1$-manifold in $\R^2$ such that some of the crossings are decorated by small circles.  Crossings are called {\it virtual crossings} when they are decorated by small circles; otherwise, they are called {\it real crossings}. 
The local moves, $R_1, R_2, R_3$, depicted in Figure~\ref{fgmoves} are flat version of classical Reidemeister moves and 
the moves, $VR_1, \dots, VR_4$, are flat version of virtual Reidemeister moves. 
We assume that virtual diagrams are oriented.  (Orientations are omitted in Figure~\ref{fgmoves}.) 
A {\it virtual doodle} is  
an equivalence class of virtual diagrams  
under an equivalence relation generated by  $R_1, R_2, VR_1, VR_2, VR_3$ and $VR_4$.

\begin{figure}[ht]
\centerline{
\begin{tabular}{ccc}
\rmoveiod{.4mm}{1}&\rmoveiiod{.35mm}{1} &\rmoveiiiod{.4mm}{1}\\
$R_1$ & $R_2$ & $R_3$\\
\end{tabular}}
\centerline{
\begin{tabular}{cccc}
\rmovevio{.4mm}{1}&\rmoveviio{.35mm}{1}&\rmoveviiio{.4mm}{1}&\rmovevivod{.4mm}{1}\\
$VR_1$ & $VR_2$ & $VR_3$ & $VR_4$ \\
\end{tabular}}
\caption{Moves}\label{fgmoves}
\end{figure}

The doodle was originally introduced by the second author and P.~Taylor in \cite{rFT} as a homotopy class of 
a collection of embedded circles in $S^2$ with no triple or higher intersection points. 
M.~Khovanov \cite{rMK} extended the idea to allow each component to be an immersed circle in $S^2$. 
It was further extended to doodles on surfaces in \cite{rBFKK}. The notion of a virtual doodle was introduced in \cite{rBFKK} and it was  shown that there is a natural bijection between the set of doodles on surfaces and the set of virtual doodles. 
For details and related topics on doodles, refer to \cite{rBFKK, rCarter91A, rCarter91B, rFT, rIT, rMK, rMant}.

We introduce an algebra, a {\it doodle switch}, which is a set  with a binary operation whose axioms correspond to 
$R_1$ and $R_2$ moves.  

For a virtual diagram $D$, the {\it fundamental doodle switch} $\mathrm{FDS}(D)$ is defined.  It is an invariant of $D$ as a virtual doodle.   
Let $T$ be a doodle switch.  
Colorings of a virtual diagram $D$ by using $T$ correspond to homomorphisms from 
$\mathrm{FDS}(D)$ to $T$.  
Thus the cardinality of colorings by $T$ is an invariant of a virtual doodle.  

For a virtual diagram $D$, we introduce a doodle switch called 
the {\it doubled fundamental doodle switch}, denoted by $\mathrm{DFDS}(D)$.  
It is also an invariant of $D$ as a virtual doodle.  
We introduce a {\it doubled coloring} of a virtual diagram $D$ by a doodle switch $T$. 
It corresponds to a homomorphism from $\mathrm{DFDS}(D)$ to $T$.  

In Section~\ref{sect:doublecover}, we introduced the notion of a {\it double covering} of a virtual diagram, which is inspired by a double covering of a virtual or twisted link diagram in \cite{rkn2, rkk7}.    
We show that if two virtual diagrams are equivalent as virtual doodles then double coverings of them are also equivalent as virtual doodles.  
It turns out that 
the doubled fundamental doodle switch of a virtual diagram $D$ is  
isomorphic to the fundamental doodle switch of a double covering $\widetilde{D}$.  

The paper is organized as follows: 
in Section~\ref{sect:doodleswitch} we define a doodle switch.  
The fundamental doodle switch and a coloring of a virtual diagram are introduced. 
In Section~\ref{sect:doubledcoloring} we define  the doubled fundamental doodle switch and 
a doubled coloring of a virtual diagram.  
In Section~\ref{sect:doublecover} we introduce a double covering $\widetilde{D}$ of a virtual diagram $D$, and   
prove that if $D$ and $D'$ are equivalent as virtual doodles then 
so are  $\widetilde{D}$ and $\widetilde{D'}$.  It is proved that the doubled fundamental doodle switch of a virtual diagram $D$ is  
isomorphic to the fundamental doodle switch of a double covering $\widetilde{D}$.

\section{Doodle switches}\label{sect:doodleswitch} 

\begin{definition}{\rm 
A {\it doodle switch} 
is a set $X$ with a binary operation $(x,y)\mapsto x \cdot y$ 
such that, for all $a, b \in X$, the following conditions are satisfied: 
\begin{itemize}
\item[(1)]
$a \cdot b = b \cdot a$ if and only if $a=b$; 
\item[(2)]
there is a unique element $u$ with $a = u \cdot b$; 
let $S: X^2 \to X^2$ be a map defined by $S(a,b) = (b \cdot a, a \cdot b)$.  Then 
$S$ is bijective. 
\end{itemize}
}\end{definition}

%
%
%
%

We call the conditions (1)--(2) {\it Axioms} of a doodle switch.  
A doodle switch $(X, \cdot)$ is denoted by $X$ when it causes no confusion.  It is often more convenient and preferred to denote $x \cdot y$ by $x^y$.  (The symbol $x^y$ is used in \cite{rFR, rFRS} for racks.)
Let $X$ be a doodle switch.
Let $\tau: X^2 \to X^2, (x,y)\mapsto (y,x)$, for $x,y \in X$, be the transposition. By definition of $S$, we have $S \circ \tau = \tau \circ S$.  

\begin{prop}
Let $X$ be a doodle switch. 
\begin{itemize}
\item[(1)] 
There exists a unique binary operation 
$\bullet: X^2 \to X$, $(x,y)\mapsto x \bullet y$ such that $S^{-1}(x, y) = (y \bullet x, x \bullet y)$ for all $x,y \in X$, i.e.,  $S(a,b)=(c,d)$ if and only if $(a,b) = (d \bullet c, c \bullet d)$. 
\item[(2)] 
There exists a unique binary operation $\cdot^{-1} : X^2 \to X, (x,y)\mapsto x \cdot^{-1} y$  such that 
$(x \cdot y) \cdot^{-1} y = (x \cdot^{-1} y) \cdot y = x$ for all $x, y \in X$. 
\item[(3)]  
There exists a unique binary operation 
$\bullet^{-1} : X^2 \to X, (x,y)\mapsto x \bullet^{-1} y$  such that 
$(x \bullet y) \bullet^{-1} y = (x \bullet^{-1} y) \bullet y = x$ for all $x, y \in X$. 

\end{itemize} 
\end{prop}

\begin{proof}
(1) Define binary operations 
$\bullet_1$ and $\bullet_2$ by  
$S^{-1}(x, y) = (y \bullet_1 x, x \bullet_2 y)$.  
Note that $S^{-1} \circ \tau = \tau \circ S^{-1}$, which implies 
$\bullet_1 = \bullet_2$.  Let $\bullet = \bullet_1 = \bullet_2$.  
If $\ast$ is a binary operation with $S^{-1}(x, y) = (y \ast x, x \ast y)$ then 
$y \ast x = y \bullet x$.  Thus $\ast = \bullet$. 

(2) For $x, y \in X$, by Axiom (2), there is a unique element $u$ with $u \cdot y=x$. Let $x \cdot^{-1} y := u$.  
Then $(x \cdot^{-1} y) \cdot y = u \cdot y =x$.   Put $v := (x \cdot y) \cdot^{-1} y$.  Then $v \cdot y = x \cdot y$. By Axiom (2), we have $v=x$.   If $\ast$ is a binary operation with 
$(x \ast y) \cdot y = x$, then $(x \ast y) \cdot y = x = (x \cdot^{-1} y) \cdot y$. 
By Axiom (2), we have $x \ast y = x \cdot^{-1} y$.  Thus $\ast = \cdot^{-1}$.  

(3) Define a binary operation $\bullet^{-1}$ by 
 $x \bullet^{-1} y = x \cdot (y \cdot^{-1} x)$ for all $x, y \in X$.   

 First we show that $(x \bullet y) \bullet^{-1} y = x$.  For $x, y \in X$, let $z, w \in X$ with $S(z,w)=(x,y)$, i.e., 
 $x = w \cdot z$, $y = z \cdot w$, $z = y \bullet x$ and $w = x \bullet y$.   
 Then $(x \bullet y) \bullet^{-1} y = w \bullet^{-1} y = w \cdot (y \cdot^{-1} w) = w \cdot z = x$.  
 
We show that $(x \bullet^{-1} y) \bullet y = x$.   For $x, y \in X$, let $u, v\in X$ with 
$u = y \cdot^{-1}x$ and $v = x \cdot u = x \cdot (y \cdot^{-1} x) = x \bullet^{-1} y$.  Then $S(u,x) =(v,y)$.  
We have $(x \bullet^{-1} y) \bullet y = v \bullet y =x$.  

If $\ast$ is a binary operation with $(x \ast y) \bullet y = x$, then 
$((x \ast y) \bullet y) \bullet^{-1} y = x \bullet^{-1} y$ and we have $x \ast y = x \bullet^{-1} y$.  Thus $\ast= \bullet^{-1}$.  
\end{proof}

Let $t: X \to X$ be a unary operation with $t(x) = x \cdot x$ for $x \in X$. 
By Axioms (1) and (2), we see that $t$ is a bijection and 
$S(a, a ) = (t(a), t(a))$ and $S^{-1}(b,b)= (t^{-1}(b), t^{-1}(b))$ for $a,b \in X$. Moreover, 
$t^{-1}(x) = x \bullet x$ for all $x \in X$. 

In what follows, when we discuss a doodle switch we assume that the operations 
$\cdot^{-1}$, $\bullet$, $\bullet^{-1}$ and $t$ 
are automatically equipped with.

For a virtual diagram $D$,  removing an open regular neighbourhood of every real crossing, we obtain a collection of oriented immersed arcs and loops in $\R^2$, whose crossings are virtual crossings of $D$.  Such an arc or loop is called a {\it semiarc} of $D$.  

The {\it fundamental doodle switch} of $D$, 
denoted by $\mathrm{FDS}(D)$,  
is the doodle switch generated by letters corresponding to the semiarcs of $D$ and the defining relations are coming from each real crossing as follows: 
Let $a, b, c, d$ be the letters corresponding to the $4$ semiarcs around a real crossing as in Figure~\ref{fgdswitch} (Left). 
The defining relations associated to the crossing are $c = b \cdot a$ and $d = a \cdot b$. 
(At a virtual crossing, two semiarcs intersect as in Figure~\ref{fgdswitch} (Right).)


\begin{figure}[ht]
\vspace{1mm}
\includegraphics[width=12cm]{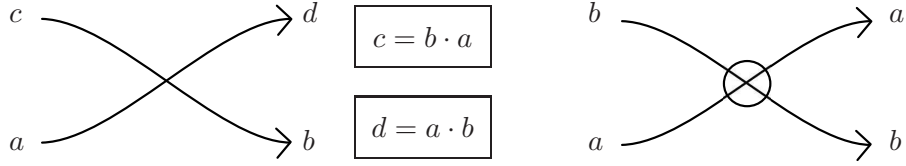}
\setlength{\unitlength}{1mm}
\begin{picture}(120,1)(0,0)
\put(0,24){$c$}   \put(39,24){$d$}  
\put(0,7){$a$}       \put(39,7){$b$}  

\put(46,18){\framebox(18,8){$c= b \cdot a$}}
\put(46,6){\framebox(18,8){$d= a \cdot b$}} 

\put(77,24){$b$}   \put(117,24){$a$} 
\put(77,7){$a$}   \put(117,7){$b$}  
\end{picture}
\caption{The defining relations associated to a crossing}\label{fgdswitch}
\end{figure}

Note that the defining relations associated to the crossing depicted in Figure~\ref{fgdswitch} (Left) imply relations  
$a = d \bullet c$ and $b= c \bullet d$, and  relations  
$a = d \cdot^{-1} b$, 
$b = c \cdot^{-1} a$, 
$c = b \bullet^{-1} d$ and 
$d = a \bullet^{-1} c$.  See Figure~\ref{fgdswitchmore}. 

\begin{figure}[ht]
\vspace{15mm}
\includegraphics[width=3.2cm, height=1.8cm]{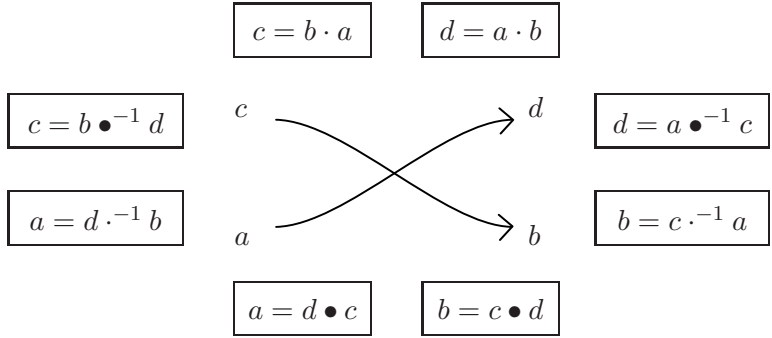}
\vspace{8mm}
\setlength{\unitlength}{1mm}
\begin{picture}(120,1)(-38,0)
\put(0,22){$c$}   \put(39,22){$d$}  
\put(0,5){$a$}       \put(39,5){$b$}  

\put(0,30){\framebox(18,7){$c= b \cdot a$}}
\put(25,30){\framebox(18,7){$d= a \cdot b$}} 

\put(0,-7){\framebox(18,7){$a= d \bullet c$}}
\put(25,-7){\framebox(18,7){$b = c \bullet d$}} 

\put(-30,18){\framebox(23,7){$c= b \bullet^{-1} d$}}
\put(-30,5){\framebox(23,7){$a= d \cdot^{-1} b$}} 

\put(48,18){\framebox(23,7){$d= a \bullet^{-1} c$}}
\put(48,5){\framebox(23,7){$b= c \cdot^{-1} a$}} 

\end{picture}
\caption{Relations around a crossing}\label{fgdswitchmore}
\end{figure}

\begin{thm}\label{thm:fundamentaldswitch}
Let $D$ and $D'$ be virtual diagrams. 
If $D$ and $D'$ are equivalent as virtual doodles, then 
$\mathrm{FDS}(D)$ and $\mathrm{FDS}(D')$  are isomorphic as virtual doodles. 
\end{thm}

\begin{proof} 
If $D$ and $D'$ are related by an $R_1$ move as in Figure~\ref{fgdswitch1} (Left), then $S(y,y)=(x,z)$. 
This implies that $x=z$ and $y = t^{-1}(x)$.  Replacing $z$ with $x$ and removing  
$y$ from the generating set of 
$\mathrm{FDS}(D')$, we obtain $\mathrm{FDS}(D)$. 
Thus $\mathrm{FDS}(D) \cong \mathrm{FDS}(D')$.  
If $D$ and $D'$ are related by an $R_1$ move as in Figure~\ref{fgdswitch1} (Right), then $S^{-1}(y,y)=(x,z)$. 
By a similar argument, we have $\mathrm{FDS}(D) \cong \mathrm{FDS}(D')$.

\begin{figure}[ht]
\vspace{3mm}
\includegraphics[width=8cm, height=2.5cm]{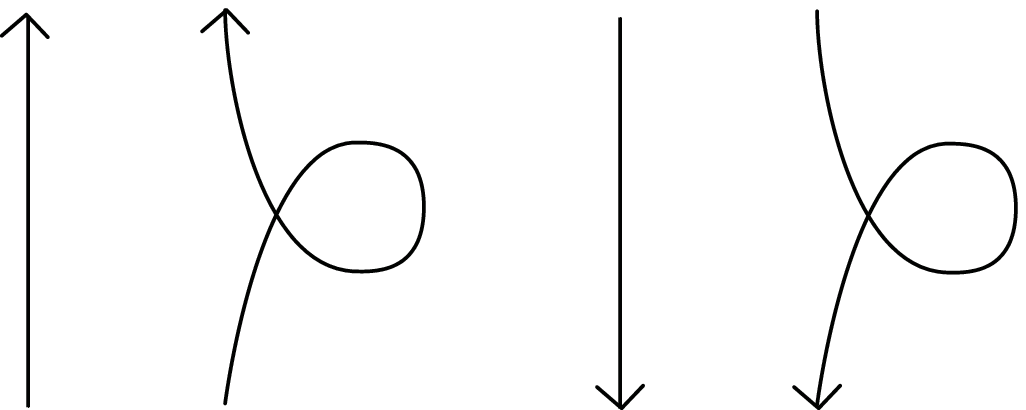}
\setlength{\unitlength}{1mm}
\begin{picture}(120,1)(0,0)

\put(18,9){${x}$}  

\put(34,9){${x}$}  
\put(46,18){${y}$} 
\put(33,25){${z}$} 
\put(16,0){ $D$ }   \put(38,0){ $D'$ }   

\put(64,25){${x}$}  

\put(80,25){${x}$}  
\put(94,18){${y}$} 
\put(80,10){${z}$}  
\put(66,0){ $D$ }   \put(84,0){ $D'$ }   

\end{picture}
\caption{}\label{fgdswitch1}
\end{figure}

If $D$ and $D'$ are related by an $R_2$ move as in Figure~\ref{fgdswitch2} (Left), then 
$S(y,w)= (x,z) =(w \cdot y, y \cdot w)$ and $S(w,v) = (z, u) = (v \cdot w, w \cdot v)$.  
This implies that 
$u=x$, $v=y$, $w= x \cdot^{-1} y$ and $z = y \cdot (x \cdot^{-1} y)$.   
Replacing $u$ with $x$ and $v$ with $y$ and removing  
$w$ and $z$ from the generating set of 
$\mathrm{FDS}(D')$, we obtain $\mathrm{FDS}(D)$. 
Thus $\mathrm{FDS}(D) \cong \mathrm{FDS}(D')$. 

If $D$ and $D'$ are related by an $R_2$ move as in Figure~\ref{fgdswitch2} (Middle), then 
$S(x,v)= (z,w) =(v \cdot x, x \cdot v)$ and $S(y,u) = (w,z) = \tau(z,w)$.  
This implies that 
$u=x$, $v=y$, $w= x \cdot y$ and $z = y \cdot x$.   
Replacing $u$ with $x$ and $v$ with $y$ and removing  
$w$ and $z$ from the generating set of 
$\mathrm{FDS}(D')$, we obtain $\mathrm{FDS}(D)$. 
Thus $\mathrm{FDS}(D) \cong \mathrm{FDS}(D')$. 

If $D$ and $D'$ are related by an $R_2$ move as in Figure~\ref{fgdswitch2} (Right), then 
$S(z,w)= (x,v)$ and $S(w,z) = (y,u)$.  
This implies that 
$u=x$, $v=y$, $w= x \bullet y$ and $z = y \bullet x$.   
Replacing $u$ with $x$ and $v$ with $y$ and removing  
$w$ and $z$ from the generating set of 
$\mathrm{FDS}(D')$, we obtain $\mathrm{FDS}(D)$. 
Thus $\mathrm{FDS}(D) \cong \mathrm{FDS}(D')$.

\begin{figure}[ht]
\includegraphics[width=12cm, height=3cm]{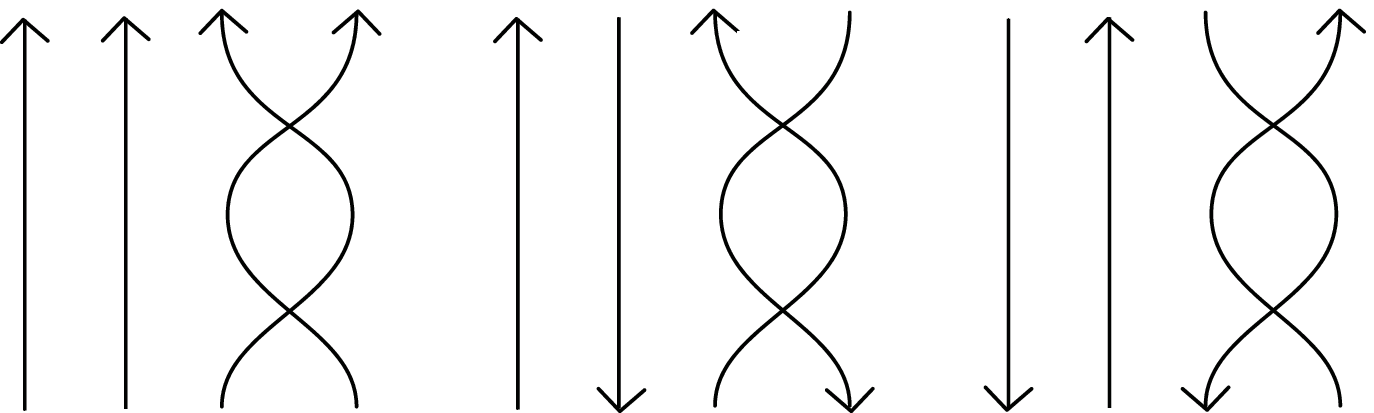}
\setlength{\unitlength}{1mm}
\begin{picture}(120,1)(0,0)
\put(-1,9){${x}$} 
\put(8,9){${y}$} 

\put(17,9){${x}$}   
\put(32,9){${y}$} 

\put(17,19){${z}$}  
\put(32,19){${w}$} 

\put(17,29){${u}$}   
\put(32,29){${v}$} 

\put(4,0){ $D$ }   \put(20,0){ $D'$ }   
\put(42,9){${x}$} 
\put(51,29){${y}$}  

\put(60,9){${x}$}   
\put(75,9){${v}$} 

\put(60,19){${z}$}   
\put(75,19){${w}$} 

\put(60,29){${u}$}  
\put(75,29){${y}$} 

\put(47,0){ $D$ }   \put(63,0){ $D'$ }   
\put(85,29){${x}$}  
\put(94,9){${y}$}   

\put(103,9){${u}$}  
\put(118,9){${y}$} 

\put(103,19){${z}$}  
\put(118,19){${w}$} 

\put(103,29){${x}$}   
\put(118,29){${v}$} 

\put(90,0){ $D$ }   \put(106,0){ $D'$ }   

\end{picture}
\caption{}\label{fgdswitch2}
\end{figure}

\end{proof}

A {\it coloring} of $D$ by a doodle switch $T$ is a  map from the set of semiarcs of $D$ to $T$ such that the relations associated to real crossings are satisfied.  
It corresponds to a homomorphism 
from the fundamental doodle switch $\mathrm{FDS}(D)$ to $T$.  
The {\it coloring number} of $D$ by $T$, denoted by $\mathrm{col}(D,T)$, is the cardinality of the set of colorings of $D$ by $T$.  

\begin{cor} 
Let $T$ be a doodle switch. Let $D$ be a virtual diagram. 
Then $\mathrm{col}(D,T)$ is an invariant of $D$ as a virtual doodle. 
\end{cor}

\begin{exa}\label{example:T4aa}
{\rm 
Let $T =(T, \cdot)$ be a doodle switch with $4$ elements $x_1, \dots, x_4$ such that the multiplication matrix $A$ of the operation $\cdot$ is given by 
$$A=
\left(
\begin{matrix}
1&2&4&3\\
3&4&2&1\\
2&1&3&4\\
4&3&1&2
\end{matrix}
\right).$$ 
Here the multiplication matrix $A$ of a binary operation $\cdot$ on $\{x_i \, | \, 1\leq i \leq n\}$ means an $n \times n$-square matrix $A =(A_{ij})$ with 
 $A_{ij}=k$ if $x_i \cdot {x_j}=x_k$ for $i, j$ $(1 \leq i, j \leq n)$. 

Let $T'$ and $T''$ be doodle switches with $3$ elements presented by the following matrices $A'$ and $A''$ respectively: 
$$
A'=
\left(
\begin{matrix}
1&2&1\\
3&3&3\\
2&1&2\\
\end{matrix}
\right) 
\quad 
\mbox{and}
\quad 
A''=
\left(
\begin{matrix}
3&2&3\\
1&1&1\\
2&3&2\\
\end{matrix}
\right).$$

Let $U$ be the unknot.  Let $d_{3,1}$ be the virtual diagram depicted in Figure~\ref{fd31}, which    
corresponds to an immersed circle in a surface with Gauss word $a b c b^{-1} a^{-1} c^{-1}$ introduced by Carter \cite{rCarter91B}.  
Let $d_{4,1}, \dots, d_{4,6}$ be virtual diagrams depicted in Figure~\ref{d41}.  

\begin{figure}[ht]
\centerline{
\includegraphics[width=1.7cm]{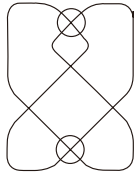} }
\caption{Carter's curve $d_{3,1}$}\label{fd31}
\end{figure}

\begin{figure}[ht]
\centerline{
\includegraphics[width=2.3cm]{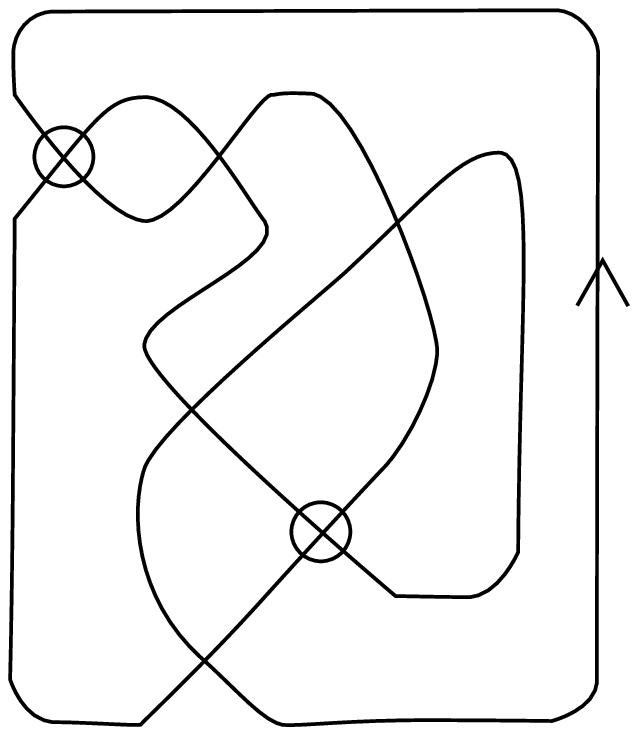}\hspace{.8cm}
\includegraphics[width=2.0cm]{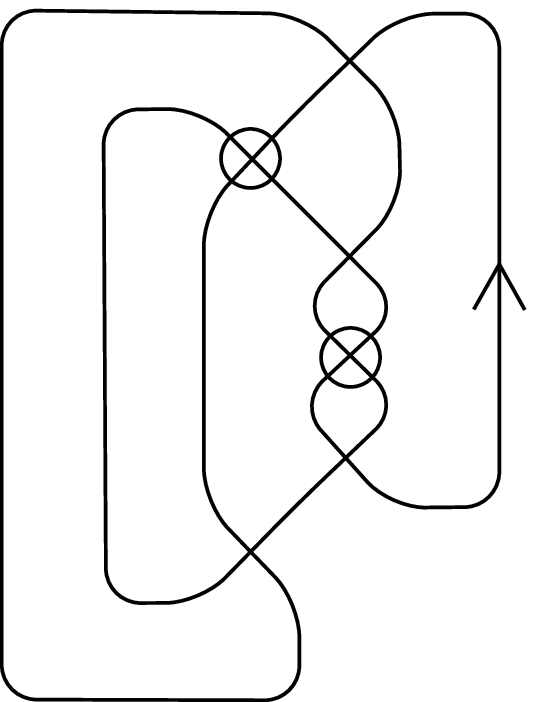}\hspace{.8cm}
\includegraphics[width=2.2cm]{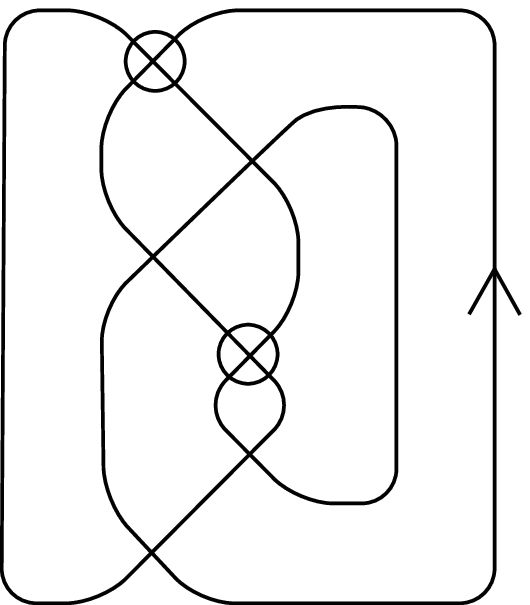}\hspace{.8cm}
}
\centerline{$d_{4,1}$\hspace{3cm}$d_{4,2}$\hspace{3cm}$d_{4,3}$}

\centerline{
\includegraphics[width=2.3cm]{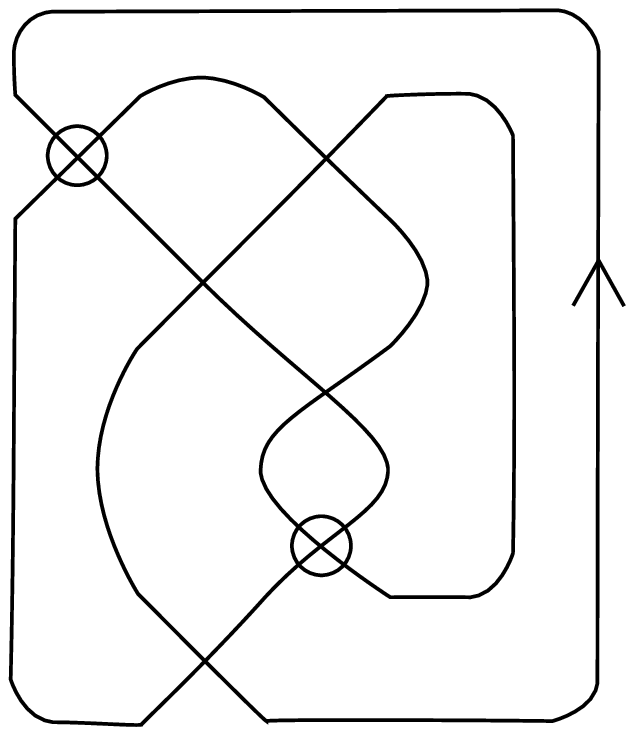}\hspace{.8cm}
\includegraphics[width=2.0cm]{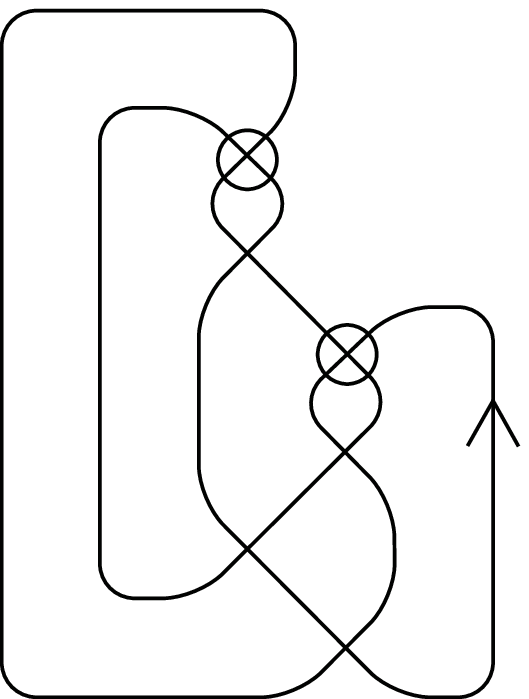}\hspace{.8cm}
\includegraphics[width=2.2cm]{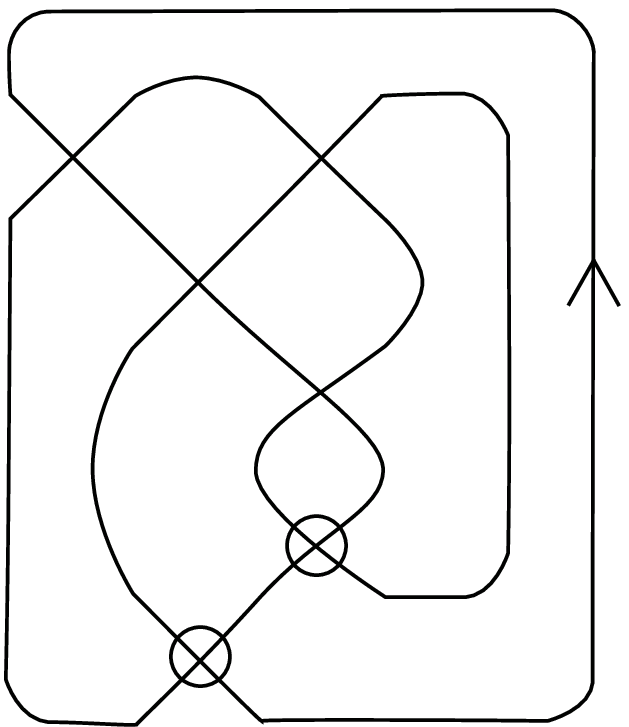}\hspace{.8cm}
}
\centerline{$d_{4,4}$\hspace{3cm}$d_{4,5}$\hspace{3cm}$d_{4,6}$}

\caption{Examples of virtual diagrams}\label{d41}
\end{figure}
 
The coloring numbers of these virtual diagrams by $T$, $T'$ and $T''$ 
are listed in Table~\ref{table:coloringnumbersT}. 

\begin{table}[htp]
\caption{Coloring numbers using $T$, $T'$ and $T''$}
\begin{center}
\begin{tabular}{c||c|c|c|c|c|c|c|c|}
& $U$ & $d_{3,1}$ & $d_{4,1}$ & $d_{4,2}$ & $d_{4,3}$ & $d_{4,4}$ & $d_{4,5}$ & $d_{4,6}$ \\
\hline
\hline
$\mathrm{col}(-,T)$ & 4 & 2 & 2 & 2 & 4 & 2 & 2 & 4\\
\hline
$\mathrm{col}(-,T')$ & 3 & 1 & 2 & 2 & 1 & 1 & 1 & 2\\
\hline
$\mathrm{col}(-,T'')$ & 3 & 1 & 1 & 1 & 1 & 1 & 1 & 2
\end{tabular}
\end{center}
\label{table:coloringnumbersT}
\end{table}%

For example, colorings of  $d_{4,1}$ by $T$, $T'$ and $T''$ are listed in 
Table~\ref{table:coloring41}, where $x_1, \dots, x_8$ are  
semiarcs of $d_{4,1}$ as in Figure~\ref{fd41ex}.  

\begin{table}[htp]
\caption{Colorings of $d_{4,1}$ by $T$, $T'$ and $T''$}
\begin{center}
\begin{tabular}{c||c|c|c|c|c|c|c|c|}
& $x_1$ & $x_2$ & $x_3$ & $x_4$ & $x_5$ & $x_6$ & $x_7$ & $x_8$ \\
\hline
\hline
$T$ & 1 & 1 & 1 & 1 & 1 & 1 & 1 & 1\\
\hline
$T$ & 3 & 3 & 3 & 3 & 3 & 3 & 3 & 3\\
\hline
\hline
$T'$ & 1 & 1 & 1 & 1 & 1 & 1 & 1 & 1\\
\hline
$T'$ & 3 & 2 & 3 & 2 & 3 & 1 & 3 & 2\\
\hline
\hline
$T''$ & 3 & 3 & 3 & 1 & 2 & 1 & 2 & 1 
\end{tabular}
\end{center}
\label{table:coloring41}
\end{table}%

\begin{figure}[ht]
\centerline{
\includegraphics[width=4.0cm]{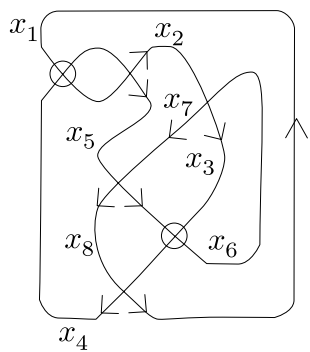} }
\caption{Colorings of $d_{4,1}$}\label{fd41ex}
\end{figure}

}\end{exa}

\begin{rem}\label{remarkA}
{\rm 
When a doodle switch $T$ satisfies an additional condition that 
$$(a \cdot b) \cdot (c \cdot b) = (a \cdot c) \cdot (b \cdot c)$$ 
for any $a, b, c \in T$, one can prove that if $D$ and $D'$ are related by an  $R_3$ move then 
$\mathrm{FDS}(D) \cong \mathrm{FDS}(D')$ and 
$\mathrm{col}(D,T) = \mathrm{col}(D',T)$.  (We omit the proof here.)

For example, the doodle switch $T$ in Example~\ref{example:T4aa} satisfies this condition, and 
$T'$ and $T''$ do not.  
The diagrams $d_{4,1}$ and $d_{4,4}$ are related by an $R_3$ move, and  they are distinguished by using 
the coloring number by $T'$. 
}\end{rem}

\section{Doubled colorings}\label{sect:doubledcoloring} 

In this section we introduce a {\it doubled coloring} of a virtual diagram $D$. We can consider some notions of \lq doubled coloring\rq. The notion we introduce here is one of them, which is inspired by a double covering of a virtual/twisted link in \cite{rkn2, rkk7}.

\begin{definition}{\rm 
The {\it doubled fundamental doodle switch} ${\mathrm{DFDS}}(D)$ 
of a virtual diagram  $D$ 
is the doodle switch 
generated by two letters $\overline{x}$ and $\underline{x}$ corresponding to every semiarc of $D$ (Figure~\ref{fg2}) with defining  relations associated to real crossings as follows. 
Let  $\overline{a},\underline{a}, \overline{b},\underline{b},\overline{c},\underline{c},\overline{d},\underline{d}$ be letters 
corresponding to the four simiarcs around a real crossing as in Figure~\ref{fgdoubledoodleswitch} (Left).  
The defining relations associated to the real crossing are 
$$
\overline{a}= \underline{d} \cdot {\underline{c}}, \quad 
\overline{b}= \underline{c} \cdot {\underline{d}}, \quad 
\overline{c}= \underline{b} \cdot {\underline{a}}, \quad \mbox{and} \quad  
\overline{d}= \underline{a} \cdot {\underline{b}}.$$  
}\end{definition}

\begin{figure}[ht]
\vspace{3mm}
\centerline{
\includegraphics[width=2.5cm]{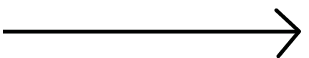}
}
\setlength{\unitlength}{1mm}
\begin{picture}(25,1)(0,0)
\put(5,9){$\overline{x}$} 
\put(5,4){$\underline{x}$} 
\end{picture}
\vspace{-1mm}
\caption{Two elements assigned to a semiarc}\label{fg2}
\end{figure}

\begin{figure}[ht]
\vspace{5mm}
\includegraphics[width=12cm]{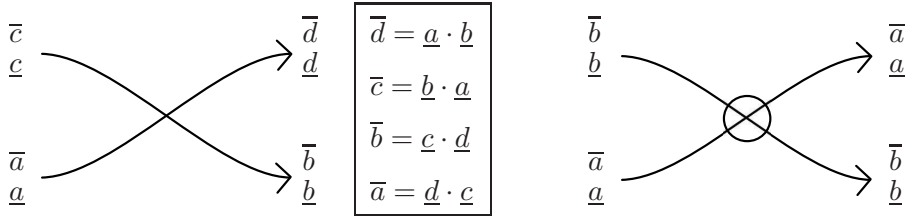}
\setlength{\unitlength}{1mm}
\begin{picture}(120,1)(0,0)
\put(0,26){$\overline{c}$}   \put(39,26){$\overline{d}$}  
\put(0,22){$\underline{c}$}   \put(39,22){$\underline{d}$}  
\put(0,9){$\overline{a}$}        \put(39,9){$\overline{b}$}  
\put(0,5){$\underline{a}$}       \put(39,5){$\underline{b}$}  

\put(48,26){{$\overline{d}= \underline{a} \cdot {\underline{b}}$}}
\put(48,19){{$\overline{c}= \underline{b} \cdot {\underline{a}}$}}
\put(48,12){{$\overline{b}= \underline{c} \cdot {\underline{d}}$}} 
\put(48,5){{$\overline{a}= \underline{d} \cdot {\underline{c}}$}} 
\put(46,3){\framebox(18,28){}}

\put(77,26){$\overline{b}$}   \put(117,26){$\overline{a}$}  
\put(77,22){$\underline{b}$}   \put(117,22){$\underline{a}$}  
\put(77,9){$\overline{a}$}        \put(117,9){$\overline{b}$}  
\put(77,5){$\underline{a}$}       \put(117,5){$\underline{b}$}  
\end{picture}
\caption{The defining relations associated to a crossing}\label{fgdoubledoodleswitch}
\end{figure}

For the sake of convenience, we call $\overline{x}$ (or $\underline{x}$) the {\it upper generator} 
(or {\it lower generator}) associated to a semiarc.

For a virtual diagram $D$ and a doodle switch  $T$, 
a {\it doubled coloring} of $D$ by $T$ is a map from the set of letters, i.e., upper and lower generators, to $T$ such that the relations associated to the real crossings are satisfied. 
It corresponds to a homomorphism from the doubled fundamental doodle switch of $D$ to $T$.  

The {\it doubled  coloring number} of $D$ by $T$ is 
the cardinality of the set of doubled colorings of $D$ by $T$, which is denoted by $\mathrm{dcol}(D,T)$.  

\begin{thm}\label{thmBFKK1}
Let $D$ and $D'$ be virtual diagrams.  If $D$ and $D'$ are equivalent as virtual doodles, then ${\mathrm{DFDS}}(D)$ 
and ${\mathrm{DFDS}}(D')$ are isomorphic as doodle switches. Consequently, the doubled coloring number by a doodle 
switch is an invariant of virtual doodles. 
\end{thm}

We prove this theorem in the next section.

\begin{exa}\label{example:dcol}{\rm 
Let $T$ be the doodle switch in Example~\ref{example:T4aa}.  
The coloring numbers and 
doubled coloring numbers of $U$, $d_{3,1}$, $d_{4,1}$, \dots, $d_{4,6}$ 
by $T$ are listed in Table~\ref{table:coloringnumbersB}. 

\begin{table}[htp]
\caption{Coloring and doubled coloring numbers by $T$}
\begin{center}
\begin{tabular}{c||c|c|c|c|c|c|c|c|}
& $U$ & $d_{3,1}$ & $d_{4,1}$ & $d_{4,2}$ & $d_{4,3}$ & $d_{4,4}$ & $d_{4,5}$ & $d_{4,6}$ \\
\hline
\hline
$\mathrm{col}(-,T)$ & 4 & 2 & 2 & 2 & 4 & 2 & 4 & 4\\
\hline
$\mathrm{dcol}(-,T)$ & 16 & 16 & 16 & 16 & 256 & 256 & 256 & 256
\end{tabular}
\end{center}
\label{table:coloringnumbersB}
\end{table}%

}\end{exa}

\begin{rem}{\rm 
Since $T$ in Example~\ref{example:T4aa} satisfies the condition in Remark~\ref{remarkA}, 
 $d_{4,1}$ and $d_{4,4}$ are not distinguished by the coloring number by $T$.  
However, $d_{4,1}$ and $d_{4,4}$ are distinguished by the doubled coloring number by $T$.  
}\end{rem}

\section{Double coverings for virtual diagrams} \label{sect:doublecover}

In this section, we introduce the notion of a double covering $\widetilde{D}$ of a virtual diagram $D$, which is inspired by a double covering of a virtual or twisted link diagram in \cite{rkn2, rkk7}.  
We show that if $D$ and $D'$ are equivalent as virtual doodles then so are their double coverings (Theorem~\ref{thmBFKK2}). 
Thus the double covering is well-defined for a virtual doodle. 
It turns out that the doubled fundamental doodle switch of $D$ is isomorphic to the fundamental doodle switch of a double covering $\widetilde{D}$ 
(Theorem~\ref{thmBFKK4}).

Let $D$ be a virtual diagram. Put two points on two semiarcs around each real crossing  of $D$ as in Figure~\ref{fgdcut}.   We call the points {\it cut points}.   The set of cut points of $D$ is denoted by $P_D$ and is called the {\it cut system} of $D$. 
(We do not introduce cut points to virtual crossings.) 

\begin{figure}[ht]
\centerline{
\includegraphics[width=1.5cm]{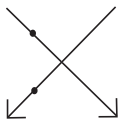}}
\caption{Cut points for a crossing}\label{fgdcut}
\end{figure}

Let $D$ be a virtual diagram $D$ and $P_D = \{p_1,\dots, p_k\}$ the cut system.  
By an isotopy of $\R^2$, we deform $D$ such that $D$ is in the half plane $x>0$ and that 
the $y$-coordinates of the cut points are all distinct.  
    Let  $D^*$ be the virtual diagram obtained by reflecting  $D$  with respect to the $y$-axis, and let 
$P_{D^*} = \{p_1^*, \dots, p_k^*\}$  be  the image of $P_D$ on $D^*$.  See Figure~\ref{fgexdbl1}.  

\begin{figure}[ht]
\centerline{\includegraphics[width=6cm]{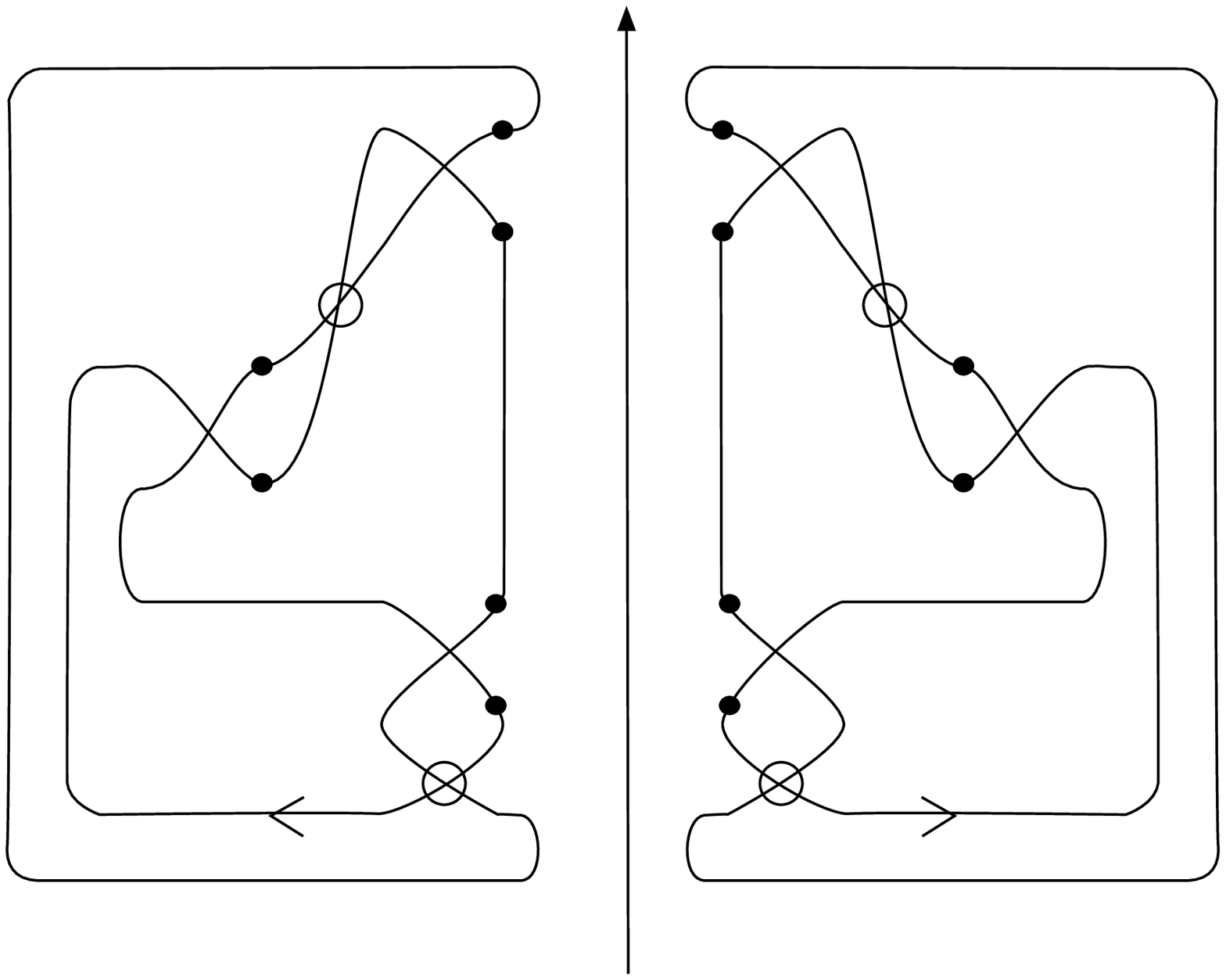}}
\centerline{$(D^*, P_{D^*})$\hspace{2cm}$(D, P_{D})$}
\caption{A virtual diagram $D$ with  cut system $P_{D}$ and the reflection images}\label{fgexdbl1}
\end{figure}
  
For each $i=1, \dots, k$,  
let $l_i$ be the horizontal line passing through $p_i$ and $p_i^*$.   
If necessary, by an isotopy of $\R^2$, we may assume that  $D \cup D^*$ intersects with 
$l_i$ transversely.  
Let $\widetilde{D}$ be a virtual diagram obtained by replacing 
$D \cup D^*$ as in Figure~\ref{fgconvert2}  in a regular neighbourhood $N(l_i)$ of $l_i$ for every $i=1, \dots, k$.  
We call the diagram $\widetilde{D}$ a {\it double covering} of $D$.  

\begin{figure}[ht]
\centerline{
\includegraphics[width=8cm]{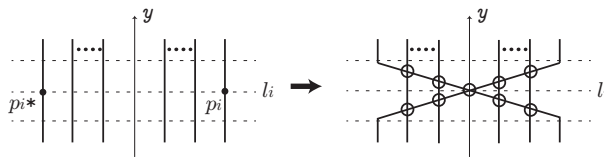}}
\caption{Replacement near cut points}\label{fgconvert2}
\end{figure}

For example, for the diagram $D$  with $P_D$ depicted in Figure~\ref{fgexdbl1}, a double covering $\widetilde{D}$  is as in Figure~\ref{fgconvert1}.

\begin{figure}[ht]
\centerline{\includegraphics[width=5.cm]{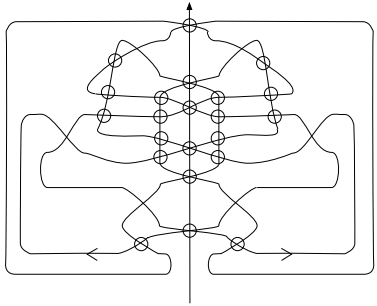}}
\caption{A double covering of a virtual diagram}\label{fgconvert1}
\end{figure}

\begin{thm}\label{thmBFKK2}
Let $D$ and $D'$ be virtual diagrams. If $D$ and $D'$ are equivalent as virtual doodles,  
then  so are $\widetilde{D}$ and $\widetilde{D'}$.  
\end{thm}

\begin{proof} 
We assume that $D$ is in the half plane $x > 0$ and the $y$-coordinates of the cut points are all distinct. 

If $D'$ is obtained from $D$ by an isotopy of $\R^2$, then $\widetilde{D}$ and $\widetilde{D'}$ are related by detour moves. 
(An sub-path of a virtual diagram which passes only through virtual crossings is called a {\it virtual path}. 
A {\it detour move} is a replacement of a virtual path with another virtual path, \cite{rBFKK}. A detour move is a consequence of $VR_1$, \dots, $VR_4$ moves.)  Thus, $\widetilde{D}$ and $\widetilde{D'}$ are equivalent as virtual doodles.

Let $D$ and $D'$ be virtual diagrams such that $D'$ is obtained from $D$ by a local move depicted in Figure~\ref{fgmoves}. 
   Let $U$ be a disc in $\R^2$ where the the local move is applied to $D$, $U^*$ the reflection image of $U$, 
and let $N$ be the smallest convex disc containing $U$ and $U^*$. 

First we consider an $R_1$ move depicted in the most left and the second left of Figure~\ref{fgDblCovFRI}.   
Figure~\ref{fgDblCovFRI}  shows the restrictions 
$(D^* \cup D) \cap N$, $(D'^* \cup D') \cap N$, $\widetilde{D} \cap N$ and $\widetilde{D'} \cap N$. 
The diagram $D^* \cup D$ and $D'^* \cup D'$ are the identical outside of $N$, and 
so are $\widetilde{D}$ and $\widetilde{D'}$.   
It is easily seen that the restriction $\widetilde{D'} \cap N$ is transformed into $(D'^* \cup D') \cap N$ by 
$VR_1$, \dots, $VR_4$ moves, since they are related by detour moves.  Then $(D'^* \cup D') \cap N$ is transformed into  $\widetilde{D} \cap N$ by two $R_1$ moves.  
Thus $\widetilde{D}$ and $\widetilde{D'}$ are equivalent as virtual doodles.  
The other cases for an $R_1$ move is similar, and we omit the proof. 

\begin{figure}[ht]
\centerline{
\begin{tabular}{cc}
\includegraphics[width=6cm]{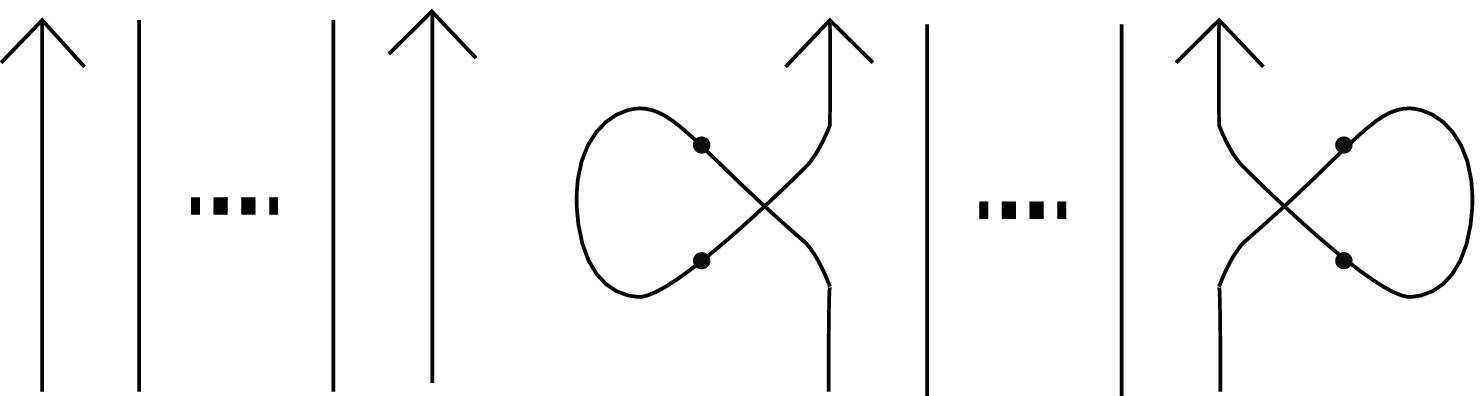}\hspace{.5cm}&\hspace{.5cm} \includegraphics[width=6cm]{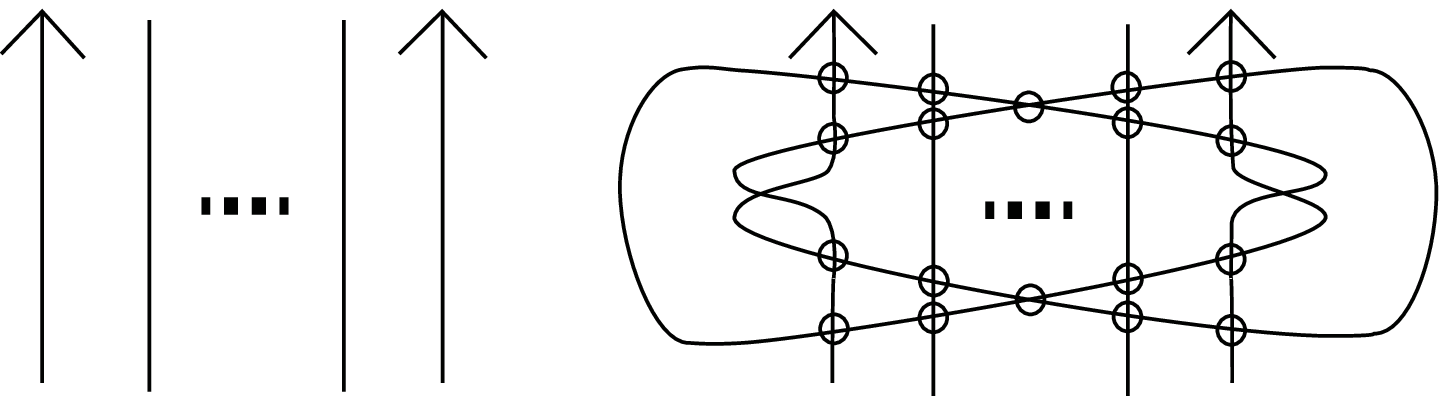}\\
$D^*\cup D$\phantom{MMMMM}$D'^*\cup D'$\phantom{MM}
&$\widetilde{D}$ \phantom{MMMMMMMM}$\widetilde{D'}$\\
\end{tabular} 
}
\caption{Double coverings of two diagrams related by an $R_1$ move}\label{fgDblCovFRI}
\end{figure}

We consider an $R_2$ move depicted in the most left and the second left (on the first row or the  second) 
of Figure~\ref{fgDblCovFRII}.  
Figure~\ref{fgDblCovFRII}  shows 
the restrictions  
$(D^* \cup D) \cap N$, $(D'^* \cup D') \cap N$, $\widetilde{D} \cap N$ and $\widetilde{D'} \cap N$.  
Remove the four real crossings from the restriction $\widetilde{D'} \cap N$ by $R_2$ moves.  
Then the result is transformed into   $\widetilde{D} \cap N$  by 
$VR_1$, \dots, $VR_4$ moves, since they are related by detour moves.  
Thus $\widetilde{D}$ and $\widetilde{D'}$ are equivalent as virtual doodles.  
The other cases for an $R_2$ move is similar, and we omit the proof. 

 \begin{figure}[ht]
 \centerline{
 \begin{tabular}{cc}
\includegraphics[width=6.5cm]{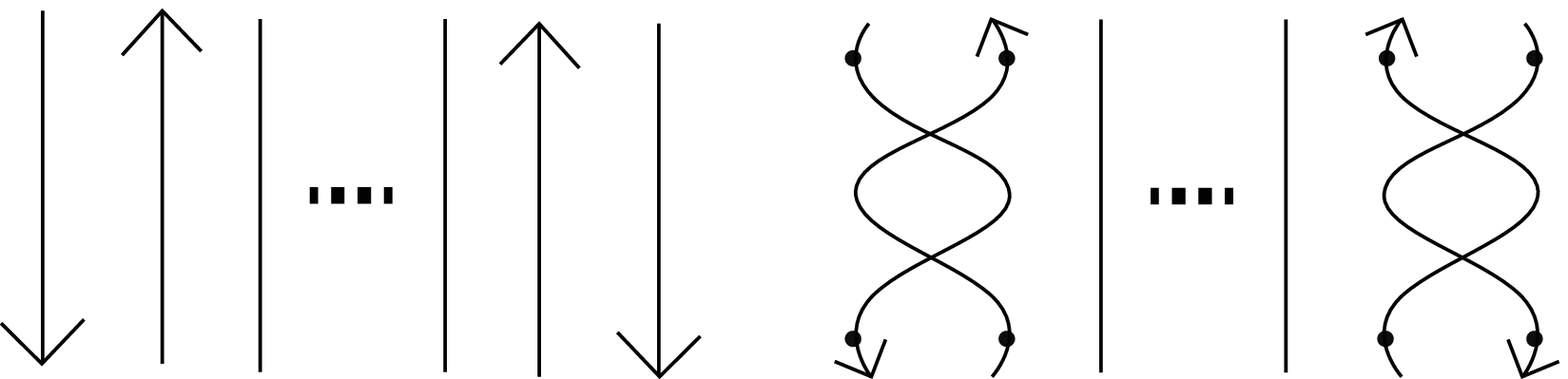}\hspace{.5cm}&\hspace{.5cm} \includegraphics[width=6.5cm]{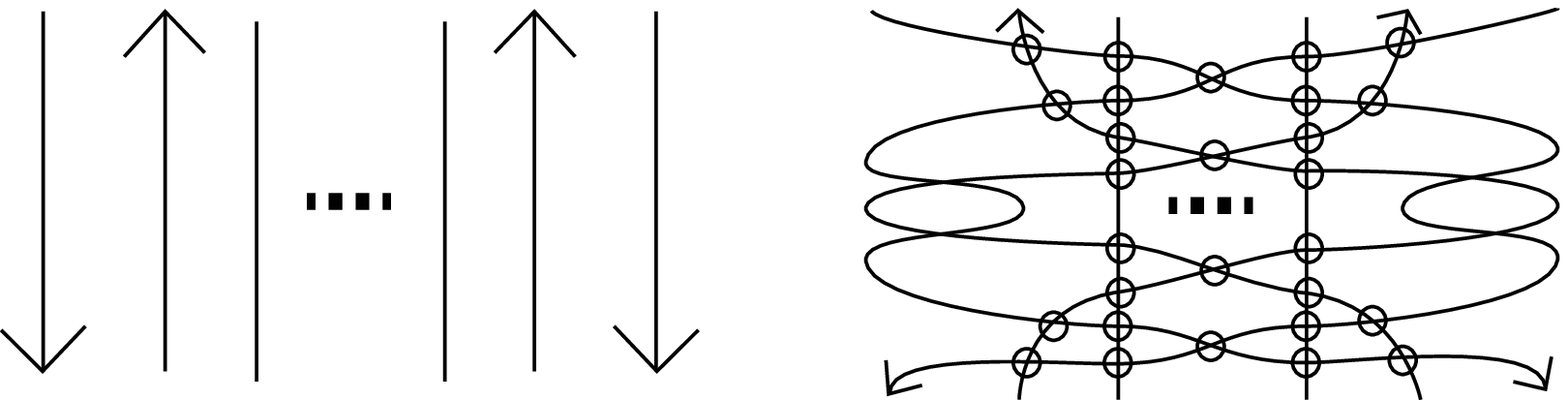}\\
$D^*\cup D$\phantom{MMMMMM}$D'^*\cup D'$
&\phantom{MM}$\widetilde{D}$ \phantom{MMMMMMMMM}$\widetilde{D'}$\\
\includegraphics[width=6.5cm]{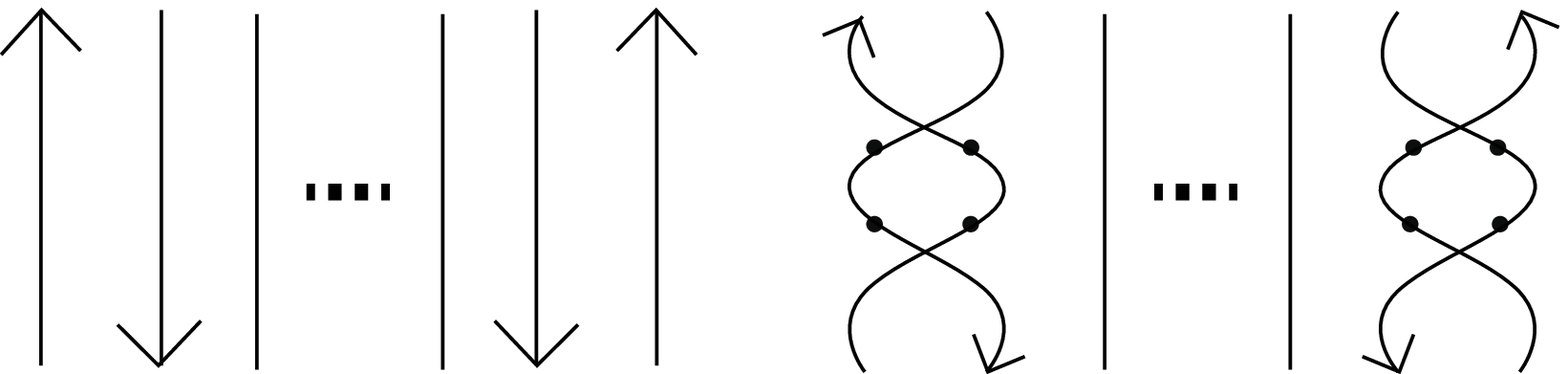}\hspace{.5cm}&\hspace{.5cm} \includegraphics[width=6.5cm]{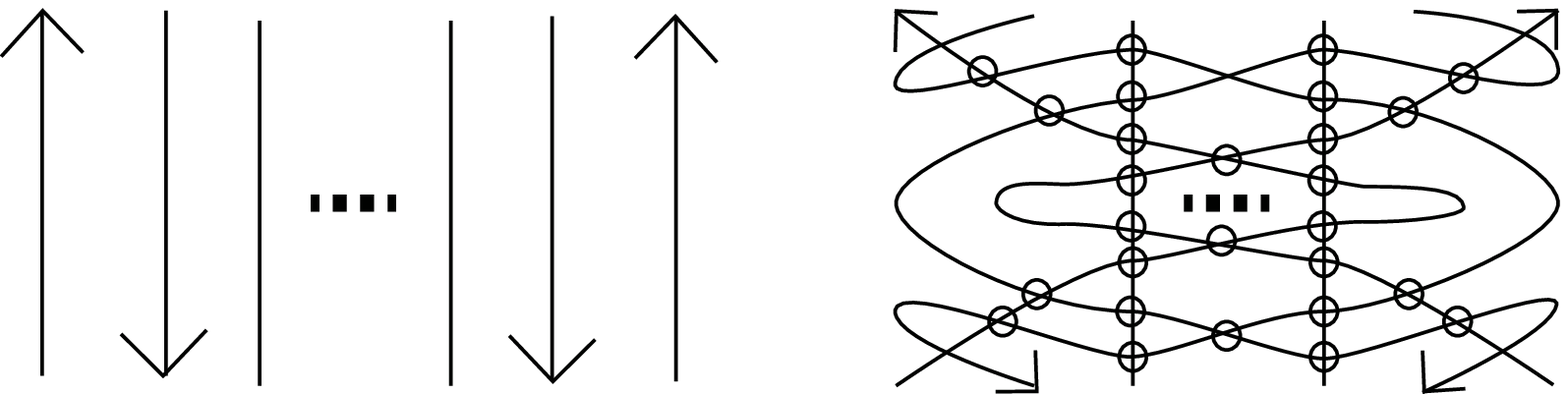}\\
$D^*\cup D$\phantom{MMMMMM}$D'^*\cup D'$
&\phantom{MM}$\widetilde{D}$ \phantom{MMMMMMMMM}$\widetilde{D'}$\\
\end{tabular} 
}
\caption{Double coverings of two diagrams related by an $R_2$ move}\label{fgDblCovFRII}
\end{figure}

When the local move is a $VR_1$, $VR_2$ or $VR_3$ move, then the restriction 
$\widetilde{D'} \cap N$ is transformed into  $\widetilde{D} \cap N$ by 
$VR_1$, \dots, $VR_4$ moves.  
Thus $\widetilde{D}$ and $\widetilde{D'}$ are equivalent as virtual doodles.  

We consider an $VR_4$ move depicted in the most left and the second left 
of Figure~\ref{fgDblCovVFRIV}.   
Figure~\ref{fgDblCovVFRIV}  shows 
the restrictions 
$(D^* \cup D) \cap N$, $(D'^* \cup D') \cap N$, $\widetilde{D} \cap N$ and $\widetilde{D'} \cap N$. 
The restriction 
$\widetilde{D'} \cap N$ is transformed into  $\widetilde{D} \cap N$ by 
$VR_1$, \dots, $VR_4$ moves, since they are related by detour moves.    
Thus $\widetilde{D}$ and $\widetilde{D'}$ are equivalent as virtual doodles.  
\end{proof}

\begin{figure}[ht]
 \centerline{
 \begin{tabular}{cc}
\includegraphics[width=6.5cm]{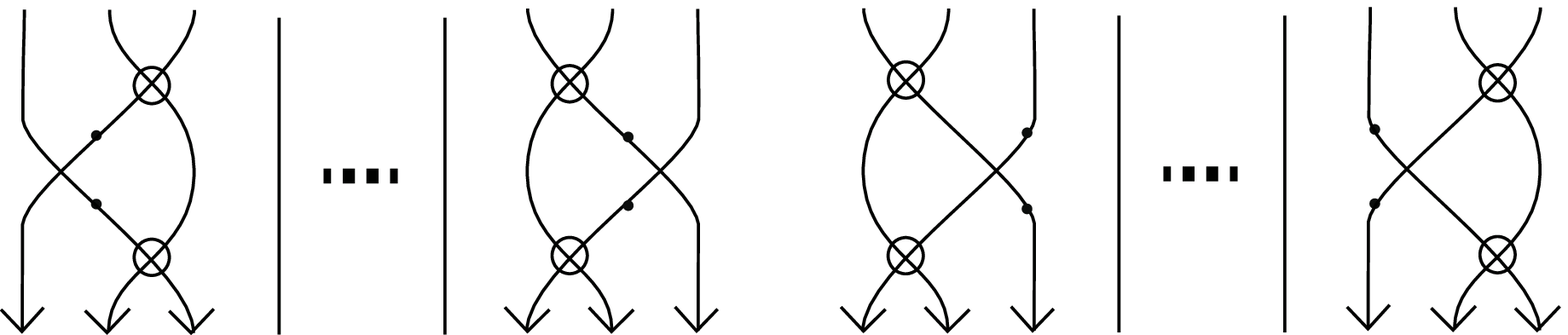}&\includegraphics[width=6.5cm]{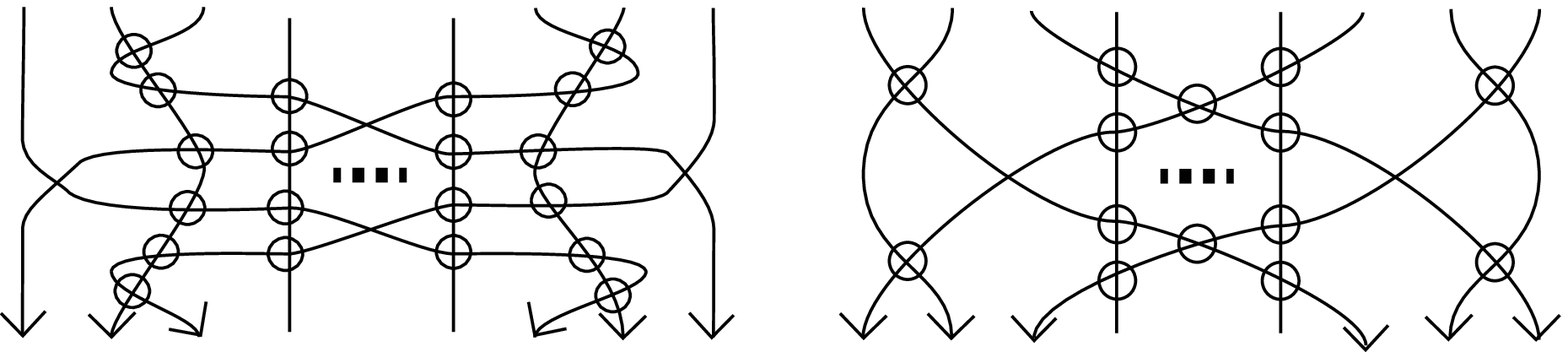}\\
$D^*\cup D$\phantom{MMMMMM}$D'^*\cup D'$
&$\widetilde{D}$\phantom{MMMMMMMM}$\widetilde{D'}$ \\
\end{tabular}}
\caption{Double coverings  of two diagrams related by a $VR_4$ move}\label{fgDblCovVFRIV}
\end{figure}

\begin{thm}\label{thmBFKK4} 
Let $D$ be a virtual diagram and let $\widetilde{D}$ be the double covering. Then 
${\mathrm{DFDS}}(D)$ and ${\mathrm{FDS}}(\widetilde{D})$ are isomorphic as doodle switches. 
\end{thm}

\begin{proof} 
Without of loss of generality, we may assume that all real crossings 
of $D$ (and $\widetilde{D}$) are oriented downward.  For each semiarc of $D$, there are two semiarcs of $\widetilde{D}$ covering  
the semiarc.  
When we label the semiarcs of $\widetilde{D}$ as in Figure~\ref{fgDblCovCrs},  
the generators and relations for the fundamental doodle switch of $\widetilde{D}$ are the same with the generators and relations for the doubled fundamental  doodle switch of $D$.   
\end{proof}

\begin{figure}[ht]
\vspace{8mm} 
\includegraphics[width=3.5cm]{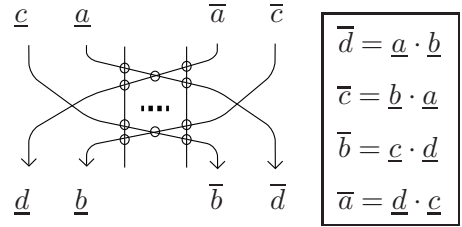} 
\setlength{\unitlength}{1mm}

\begin{picture}(120,10)(-46,0)
\put(-5,30){$\underline{c}$}  \put(3,30){$\underline{a}$}  
\put(21,30){$\overline{a}$}  \put(29,30){$\overline{c}$} 

\put(-5,5){$\underline{d}$}  \put(3,5){$\underline{b}$}  
\put(21,5){$\overline{b}$}  \put(29,5){$\overline{d}$} 

\put(38,26){{$\overline{d}= \underline{a} \cdot {\underline{b}}$}}
\put(38,19){{$\overline{c}= \underline{b} \cdot {\underline{a}}$}}
\put(38,12){{$\overline{b}= \underline{c} \cdot {\underline{d}}$}} 
\put(38,5){{$\overline{a}= \underline{d} \cdot {\underline{c}}$}} 
\put(36,3){\framebox(18,28){}}

\end{picture}
\caption{Generators and relations  arising from the double covering}\label{fgDblCovCrs}
\end{figure}

Theorem~\ref{thmBFKK1} follows from Theorems~\ref{thm:fundamentaldswitch},~\ref{thmBFKK2} and \ref{thmBFKK4}.

 \end{document}